\documentclass[10pt]{amsart}
\usepackage[dvips]{epsfig}

\newtheorem{theorem}{Theorem}

\newtheorem{proposition}[theorem]{Proposition}

\theoremstyle{definition}

\usepackage{amscd,amssymb}
\usepackage{youngtab}
\usepackage{rotating}

\begin{document}

\title[The Bulgarian solitaire and the mathematics around it]
{The Bulgarian solitaire\\
and the mathematics around it}

\author[Vesselin Drensky]
{Vesselin Drensky}
\address{Institute of Mathematics and Informatics,
Bulgarian Academy of Sciences,
Acad. G. Bonchev Str., Block 8, 1113 Sofia, Bulgaria}
\email{drensky@math.bas.bg}

\thanks
{This project was partially supported by Grant I 02/18
``Computational and Combinatorial Methods
in Algebra and Applications''
 of the Bulgarian National Science Fund.}


\begin{abstract}
The Bulgarian solitaire is a mathematical card game played by one person.
A pack of $n$ cards is divided into several decks (or ``piles'').
Each move consists of the removing of one card from each deck
and collecting the removed cards to form a new deck.
The game ends when the same position occurs twice.
It has turned out that when $n=k(k+1)/2$ is a triangular number,
the game reaches the same stable configuration with size of
the piles $1,2,\ldots,k$. The purpose of the paper is to tell
the (quite amusing) story of the game and to discuss mathematical
problems related with the Bulgarian solitaire.
\end{abstract}

\maketitle

\begin{center}
Dedicated to the memory of Borislav Bojanov (1944--2009),\\
a great mathematician, person, and friend.
\end{center}
\section{The story}
The popularity of the Bulgarian solitaire started around 1980. Below we present the version
of Borislav Bojanov \cite{B2} who is one of the main protagonists in the story.

The problem was brought to Bulgaria by the famous number theorist Anatolii Karatsuba
from the Steklov Mathematical Institute in Moscow.
In May 1980 he visited
the Institute of Mathematics and Informatics at the Bulgarian Academy of Sciences in Sofia.
Once, after his lecture at the Seminar of Approximation Theory, he told his Bulgarian colleagues the story of the problem.

Konstantin Oskolkov, in that time professor at the Steklov Institute,
was traveling from Moscow to Leningrad (now Saint Petersburg)
in the night, by the  fastest train in the Soviet Union, the so called ``Red Arrow''.
There was another man in his compartment and they started a
conversation. When the other man learned that Konstantin Oskolkov is a mathematician,
he showed him the following game.

{\it A pack of $n=1+2+\cdots+k$ cards is divided in an arbitrary way in several packs.
Each move consists of the removing of one card from each deck
and collecting the removed cards to form a new deck. Surprisingly, it turns out that after several moves
one reaches the stable position of $k$ piles consisting of $1,2,\ldots,k$ cards, respectively.}
(The legend claims that the game was illustrated with several experiments with 15 cards.)

For example, starting with a deck of 10 cards divided in three packs of size 4, 3, 3, as in Fig. 1,
we obtain a new pack of 3 cards and the number of cards in the old packs decreases to 3, 2, 2, respectively.
\begin{figure}[htb]
\begin{center}
\epsfxsize=12cm
$$\epsfbox{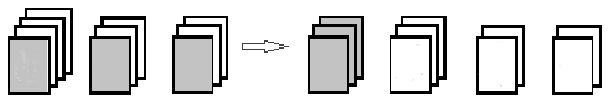}$$

Fig. 1. A deck of 10 cards is divided in three packs of size 4, 3, 3.
\end{center}
\end{figure}
It is more convenient to denote only the size of the packs, ordering the sizes in nonincreasing order.
For example, starting from the position $(4,3,3)$, we have marked the new size of the new pack in bold and
have consecutively
\[
(4,3,3)\Rightarrow (\text{\bf 3},3,2,2)\Rightarrow (\text{\bf 4},2,2,1,1)\Rightarrow (\text{\bf 5},3,1,1)
\]
(In the last step the two packs consisting of a single card disappear.) Then we continue
\[
(5,3,1,1)\Rightarrow (\text{\bf 4},4,2)\Rightarrow (\text{\bf 3},3,3,1)\Rightarrow (\text{\bf 4},2,2,2)
\]
\[
\Rightarrow (\text{\bf 4},3,1,1,1)\Rightarrow (\text{\bf 5},3,2)\Rightarrow (4,\text{\bf 3},2,1)\Rightarrow(\text{\bf 4},3,2,1).
\]
In this way we obtain the stable position $(4,3,2,1)$.

Returning back to Moscow, Konstantin Oskolkov told the problem to the people of the Department of Number Theory
at the Steklov Institute. Anatolii Karatsuba described this moment in the following way. ``When
Genadii Arkhipov (professor in Number Theory  who liked very much
nice problems) learned about the problem, his face took a Satanic expression, he
ran to his office, closed the door and did not came out until he  solved
the problem.''

Borislav Bojanov also liked very much nice problems. He went home, waited until the
children went to the bed and then started to think about it. Around
midnight he found a solution and was very happy. The next day he shared the
solution with some of his colleagues. Pencho Petrushev said that he also
had a solution. Milko Petkov who was an editor of the Bulgarian high school
mathematical journal ``Obuchenieto po matematika'' (``Education in Mathematics'')
published the problem in the section ``Competition Problems'' in the issue 5 of
1980. Since no student submitted a solution, in 1981 the Editorial Board of the
journal decided to publish the solution of Borislav Bojanov \cite{B1}.

Approximately in the same time the problem was published by
S. Limanov and A. L. Toom in the issue 11 of 1980 of the Russian mathematical journal ``Kvant''.
The solution of Toom \cite{T} appeared also in 1981. It contains also some analysis of the general case
of an arbitrary number of cards.
It seems that \cite{B1} and \cite{T} are the first published solutions of the Bulgarian
solitaire.

In that time the Swedish mathematician Gert Almkvist from the University of Lund visited
the Department of Algebra at the Institute of Mathematics and Informatics in Sofia. When he learned the problem
he brought it to Sweden and told it to his colleagues including his friend Henrik Eriksson from
the Royal Institute of Technology in Stockholm. In 1981 Eriksson wrote the paper \cite{E} where he also presented a solution
for the puzzle and gave it the name
{\it Bulgarian solitaire} ({\it Bulgarisk patiens} in Swedish). Later he visited the USA and spread the puzzle there.
J{\o}rgen Brandt from the Aarhus University, Denmark, also learnt
about the problem but without its name, and in 1982 published another solution \cite{Br}, where he also analyzed the general case.
(Brandt starts his paper with ``The problem to be discussed in the following has been circulating for some time.'')
In 1982 Donald Knuth used the Bulgarian solitaire to start his
Programming and Problem-Solving Seminar in Stanford \cite{HKn}. Finally, with the help of Ron Graham the problem reached Martin Gardner who included it
in his paper \cite{G}. The paper by Gardner was the starting point of the popularity of the Bulgarian solitaire among mathematicians
all over the world and was the main source of references for many years. For already 35 years the Bulgarian solitaire and its generalizations
continue to inspire new research in combinatorics, game theory, probability, computer science, and to be an object of intensive study in
research and teaching literature. Due to the efforts of Henrik Eriksson in 2005
and the paper by Brian Hopkins \cite{H1} in 2012, recently the real story of the Bulgarian solitaire finally reached the large audience.

\section{The solution}

In the first publications on the Bulgarian solitaire \cite{B1, T, E}
the main problem is stated in three different ways. In the Bulgarian version \cite{B1} there are $k(k+1)/2$ balls grouped in $m$ piles.
In the Russian version \cite{T} a clerk from the Circumlocution Office\footnote{The Circumlocution Office is a place of endless confusion
in {\it Little Dorrit} by Charles Dickens.}
rearranges piles of volumes of Encyclop{\ae}dia Britannica.
The Swedish text \cite{E} handles packs of cards. Nevertheless the three solutions use similar ideas.
An exposition of Toom's proof \cite{T} with more details can be found in \cite{HKK}.

As we already mentioned, instead of considering packs of cards, we may consider the sequence of the number of cards in each pack.
Since we are not interested in the order of the packs, we may order the integers in the sequence in nonincreasing way. A finite sequence of
nonnegative integers
\[
\lambda=(\lambda_1,\ldots,\lambda_c),\quad \lambda_1\geq\ldots\geq\lambda_c\geq 0,\quad \lambda_1+\cdots+\lambda_c=n,
\]
is called a {\it partition} of $n$. (The standard notation is $\lambda\vdash n$.)
The partition $\lambda=(\lambda_1,\ldots,\lambda_c)$ is visualized by its {\it Young diagram} $[\lambda]$ (also called {\it Ferrers diagram}
when represented using dots)
consisting of boxes arranged in left-justified rows, with $\lambda_i$ boxes in the $i$-th row. For example, the Young diagram of the partition
$\lambda=(4,3,3)\vdash 10$ is in Fig. 2.

\begin{center}
\yng(4,3,3)\\
Fig. 2. $[\lambda]=[4,3,3]$
\end{center}

\noindent For our purposes it is more convenient to rotate the Young diagram on $90^{\circ}$, when the height of each row is equal to
the number of cards in the corresponding pack, see Fig. 3.
\begin{figure}[htb]
\begin{center}
$$
\epsfbox{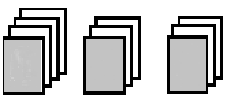}\quad\stackrel{\Rightarrow}{\phantom{\Downarrow}}
\quad\stackrel{(4,3,3)}{\phantom{\Downarrow}}
\quad\stackrel{\Rightarrow}{\phantom{\Downarrow}}\quad \yng(4,3,3)
\quad \stackrel{\Rightarrow}{\phantom{\Downarrow}}\quad\yng(1,3,3,3)
$$
Fig. 3.
\end{center}
\end{figure}
\noindent Then the move in the Bulgarian solitaire consists of removing the bottom row of the (rotated) Young diagram
and adding it as a column, as shown in Fig. 4.
\begin{figure}[htb]
\begin{center}
$$
\epsfbox{10_cards_bis.eps}\quad\stackrel{\Rightarrow}{\phantom{\Downarrow}}\quad
\young(\:,\:\:\:,\:\:\:,\times\times\times)
\quad \stackrel{\Rightarrow}{\phantom{\Downarrow}}\quad
\young(\times\:,\times\:\:\:,\times\:\:\:)
\quad \stackrel{\Rightarrow}{\phantom{\Downarrow}}\quad
\epsfbox{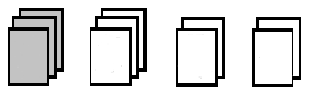}
$$
Fig. 4.
\end{center}
\end{figure}
\noindent In the language of partitions, we start with a partition $\lambda=(\lambda_1,\ldots,\lambda_c)$ with $\lambda_c>0$
and obtain the partition ${\mathcal B}(\lambda)=(c,\lambda_1-1,\ldots,\lambda_c-1)$. Clearly, if
$\lambda_i-1>c\geq \lambda_{i+1}-1$, as in Fig. 5, then we assume that
${\mathcal B}(\lambda)=(\lambda_1-1,\ldots,\lambda_i-1,c,\lambda_{i+1}-1,\ldots,\lambda_c-1)$.
\begin{center}
\[
\young(\cdot,\cdot,\cdot,\cdot\:,\cdot\:,\times\times\times)
\quad \stackrel{\Rightarrow}{\phantom{\Downarrow}}\quad
\young(\times,\times,\times)\young(\cdot,\cdot,\cdot,\cdot\:,\cdot\:)
\quad \stackrel{\Rightarrow}{\phantom{\Downarrow}}\quad
\young(\cdot,\cdot,\cdot\times,\cdot\times\:,\cdot\times\:)
\]
Fig. 5.
\end{center}
This is a typical example of a {\it discrete dynamical system}.
We consider the set ${\mathcal P}(n)$ of all partitions of $n$ and the operator $B:{\mathcal P}(n)\to{\mathcal P}(n)$
which in a period of time changes the state of the system, the partition $\lambda$, to the new state, the partition ${\mathcal B}(\lambda)$.
Hence  $\mathcal B$ plays the role of the {\it updating function of the system}.
The main problem is, starting with the initial state $\lambda^{(0)}$ to determine the state of the system
$\lambda^{(t)}={\mathcal B}(\lambda^{(t-1)})$ which it will reach after
some interval of time $t$. Since we have a finite number of states ${\mathcal P}(n)$ only, we may associate to the discrete dynamical system
its oriented graph with vertices the partitions $\lambda$ of $n$ and oriented edges $(\lambda,{\mathcal B}(\lambda))$. In Fig. 6 we give the graph
for $n=6$ and $n=7$.

\begin{figure}[htb]
\begin{center}
$${\epsfxsize=5cm\epsfbox{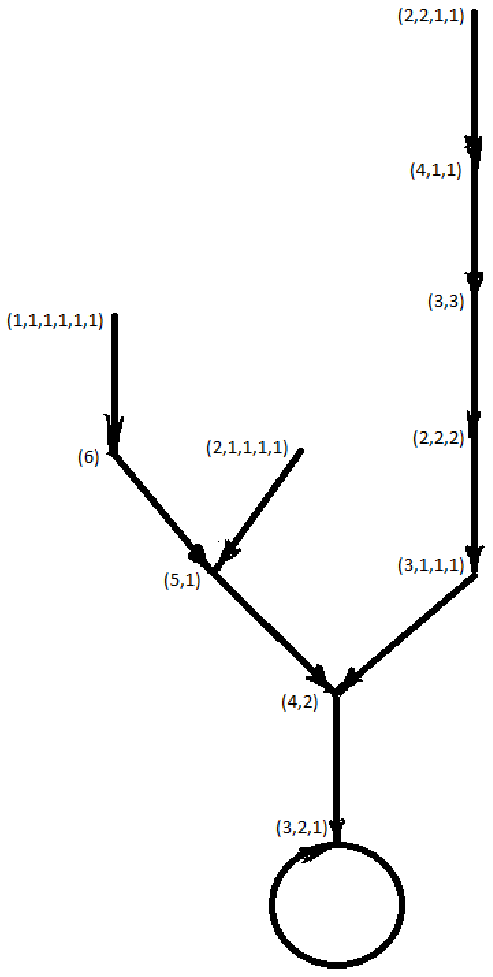}}\quad {\epsfxsize=8cm\epsfbox{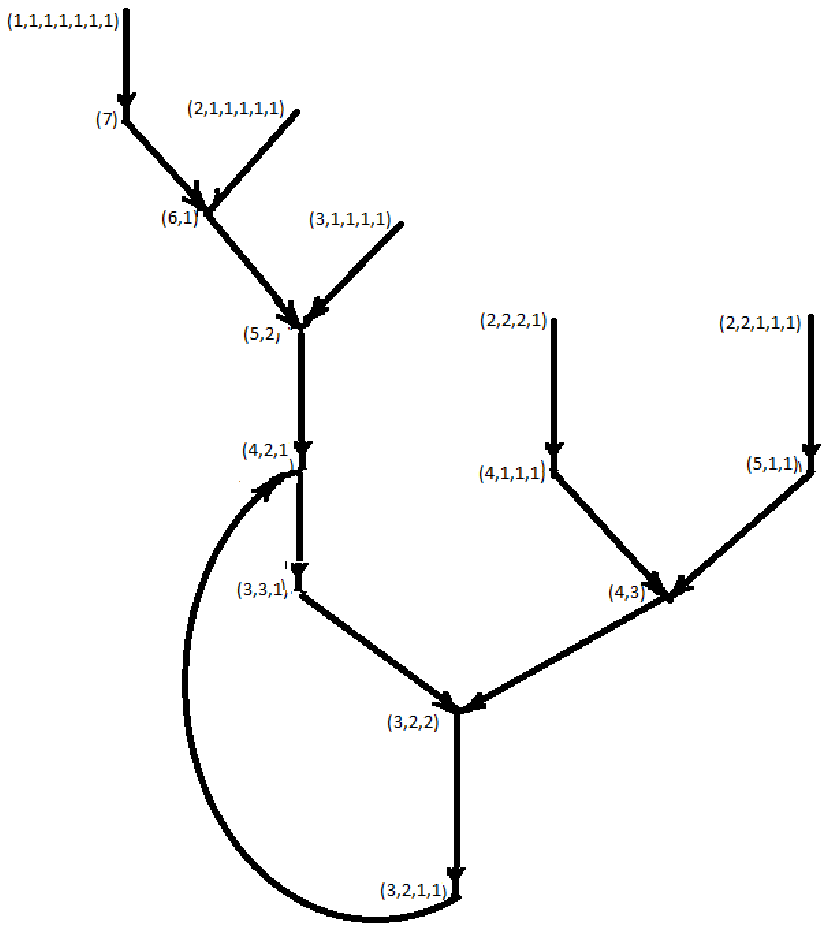}}$$

Fig. 6. The graph for the Bulgarian solitaire for $n=6$ and $n=7$.
\end{center}
\end{figure}

We shall present the solution of the Bulgarian solitaire from
\cite{T} modified in the spirit of the exposition in \cite{Br} and the solution proposed by Anders Bj\"orner, according to
the student essays \cite{O, Gr}. There are also several other solutions, using different arguments, see, e.g., the inductive proof of
Me\v{s}trovi\'c \cite{Me}.

\begin{theorem}\label{theorem of Bulgarian solitaire}
When the total number $n=k(k+1)/2$ of cards is triangular, the Bulgarian solitaire will converge
into piles of size $1, 2,\ldots,k$.
\end{theorem}

\begin{proof}
We use the brilliant visualization of the Bulgarian solitaire, {\it the cradle model}, suggested by Bj\"orner.
Let $\lambda=(\lambda_1,\ldots,\lambda_c)\vdash n$, $\lambda_c>0$, be the partition corresponding to the given collection of card packs.
We turn the Young diagram $[\lambda]$ counter-clockwise by $45^{\circ}$, see Fig. 7,
and further consider this $45^{\circ}$-turn interpretation of $[\lambda]$.
\begin{figure}[htb]
\begin{center}
\[
\epsfxsize=3cm
\epsfbox{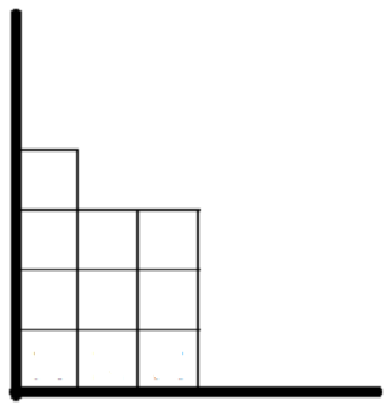}
\stackrel{\Rightarrow}{\begin{matrix}\phantom{\Downarrow}\\
\phantom{\Downarrow}\\
\phantom{\Downarrow}\\
\end{matrix}}
\begin{turn}{45}
\epsfxsize=3cm
\epsfbox{rotate_0.eps}
\end{turn}
\]
Fig. 7. $\lambda=(4,3,3)$.
\end{center}
\end{figure}

\noindent
Assuming that the boxes of $[\lambda]$
are material points with the same mass $m$, we consider the potential energy of the system
\[
U(\lambda)=mg\sum h_{ij},
\]
where $g\approx 9,8 \text{ m/s}^2$ is the free fall acceleration on Earth,
the sum runs on all boxes of $[\lambda]$,
and $h_{ij}$ is the height of the center of the box with coordinates $(i,j)$ corresponding to the $j$-th card of the $i$-th pack.
Clearly, $h_{ij}$ is proportional to $i+j$ and we may assume that it is equal to $i+j$ units.
As we discussed above, the move of the Bulgarian solitaire removes the $c$ boxes of the bottom row of $[\lambda]$
and adds them as the first column, as shown in Fig. 4. In this way, the box $(j,1)\in [\lambda]$ becomes the box $(1,j)\in[{\mathcal B}(\lambda)]$.
Obviously, in the $45^{\circ}$-turn interpretation the potential energy of these $c$ boxes of $[\lambda]$ does not change.
The other $n-c$ boxes of $[\lambda]$ move one step to the right, from the position $(i,j)$, $j>1$, to the position $(i+1,j-1)$.
Hence they also preserve their potential energy.
Therefore, if $c\geq \lambda_1-1$, as in Fig. 8, the move of the Bulgarian solitaire
forces the boxes to cycle on the same level and
preserves the potential energy of the Young diagram.
\begin{figure}[htb]
\begin{center}
\[
\begin{turn}{45}
\epsfxsize=3cm
\epsfbox{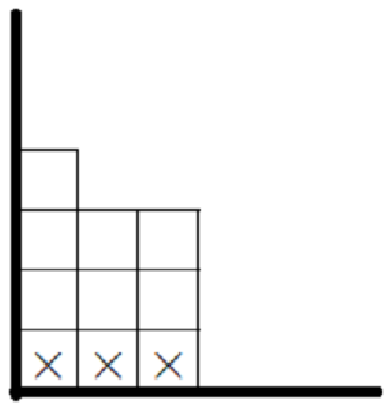}
\end{turn}
\stackrel{\Rightarrow}{\begin{matrix}\phantom{\Downarrow}\\
\phantom{\Downarrow}\\
\phantom{\Downarrow}\\
\end{matrix}}
\begin{turn}{45}
\epsfxsize=3cm
\epsfbox{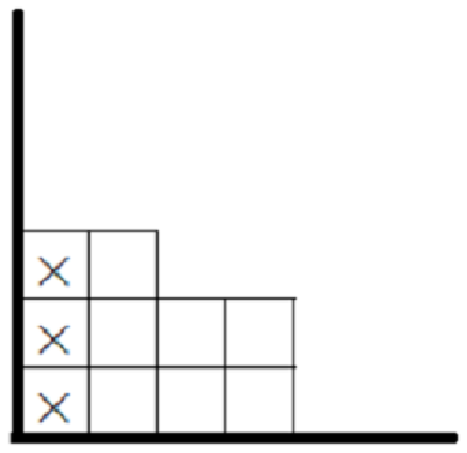}
\end{turn}
\]
Fig. 8. ${\mathcal B}(4,3,3)=(3,3,2,2)$.
\end{center}
\end{figure}
\noindent If $c<\lambda_1-1$, as in Fig. 9, then by the gravity the excessive boxes of the second pile will
fall down southwest and the potential energy of the Young diagram will decrease.
\begin{figure}[htb]
\begin{center}
\[
\begin{turn}{45}
\epsfxsize=2cm
\epsfbox{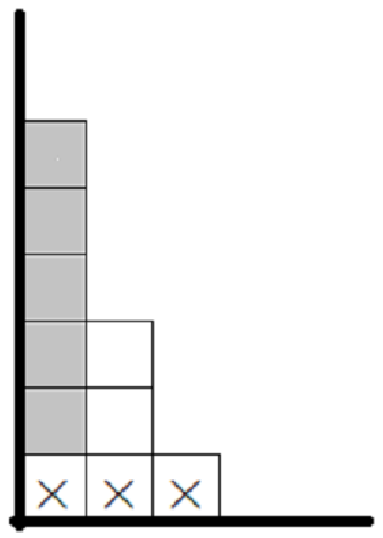}
\end{turn}
\stackrel{\Rightarrow}{\begin{matrix}\phantom{\Downarrow}\\
\phantom{\Downarrow}\\
\phantom{\Downarrow}\\
\end{matrix}}
\begin{turn}{45}
\epsfxsize=2cm
\epsfbox{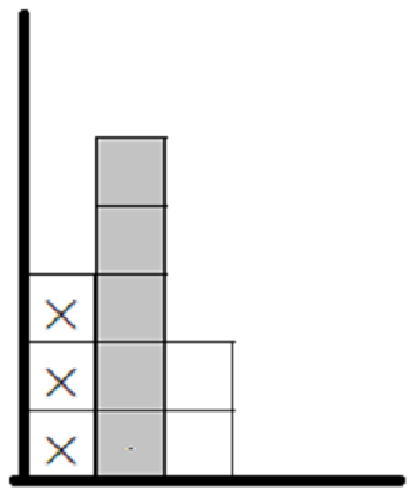}
\end{turn}
\stackrel{\Rightarrow}{\begin{matrix}\phantom{\Downarrow}\\
\phantom{\Downarrow}\\
\phantom{\Downarrow}\\
\end{matrix}}
\begin{turn}{45}
\epsfxsize=2cm
\epsfbox{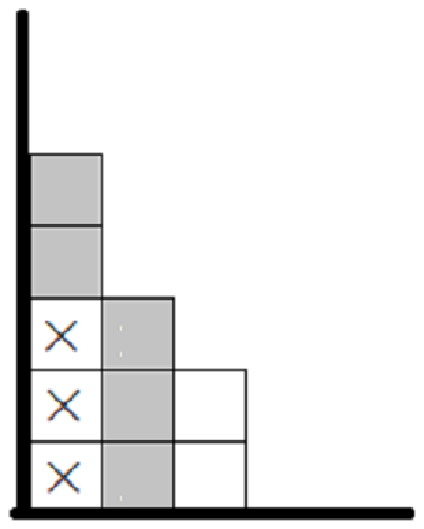}
\end{turn}
\]
Fig. 9. ${\mathcal B}(6,3,1)=(3,5,2)=(5,3,2)$.
\end{center}
\end{figure}
\noindent Since the partitions of $n$ are a finite number, we shall reach the position when the moves do not
decrease the height of the boxes and the potential energy of the system.
In this moment, let $r$ be the maximal integer with the property that
the first $r$ levels of the Young diagram consisting of boxes $(i,j)$ of height
$i+j=2,3,\ldots,r+1$, do not have empty places. Hence these $r$ levels contain $1+2+\cdots+r=r(r+1)/2$ boxes.
If $r=k$, then we have reached the stable position $(k,k-1,\ldots,2,1)$. Otherwise, $n-r(r+1)/2=k(k+1)/2-r(r+1)/2\geq k>r$.
Hence, the $(r+1)$-th level has an empty place and the $(r+2)$-nd level contains at least one box. Each move pushes this box
one place to the right and, when it reaches the most right position, it starts again from the most left position. The period of the
repetition of the positions is $r+2$.
Hence after several moves this box will be in position $(2,r+1)$. Now we follow the position of one of the empty places in level $r+1$.
It also moves to the right with period $r+1$. Since the integers $r+1$ and $r+2$ are relatively prime,
 in the moments $0, (r+2),2(r+2),\ldots,r(r+2)$ the empty place with be in
pairwise different places. In some moment it will be in the most left position $(1,r+1)$, as in Fig. 9.
Therefore the box in position $(2,r+1)$ will move to the empty position $(1,r+1)$, decreasing the potential energy of the system,
which is a contradiction. Hence the minimal potential energy is reached in the stable position $(k,k-1,\ldots,2,1)$ only.
\end{proof}

Now we shall consider the general case of any $n$. The first considerations are in \cite{T}, the detailed study was done in \cite{Br}.
Consider the oriented graph associated with the set ${\mathcal P}(n)$ of all partitions of $n$
with vertices the partitions $\lambda\in{\mathcal P}(n)$ and oriented edges $(\lambda,{\mathcal B}(\lambda))$. Clearly, the graph
consists of several components and, starting from any vertex of a given component, the multiple application of the operator $\mathcal B$
defines a cycle which is unique for the component. In Fig. 6 the cycle of the unique component of the graph of ${\mathcal P}(7)$ is
\[
\lambda=(4,2,1)\stackrel{\mathcal B}{\to}(3,3,1)\stackrel{\mathcal B}{\to}(3,2,2)\stackrel{\mathcal B}{\to}(3,2,1,1)\stackrel{\mathcal B}{\to}
{\mathcal B}^4(\lambda)=\lambda=(4,2,1).
\]

\begin{theorem}\label{The general case of the solitaire}
Let $n$ have the form $n=(k-1)k/2+r$, $0<r\leq k$, and let $\lambda\in{\mathcal P}(n)$. Then in the interpretation of the cradle model
the solitaire will converge with a cycle of partitions which consists of the triangular partition as bottom and $r$ surplus blocks cycling above.
The number of the components of the oriented graph associated with the partitions of $n$
is equal to the number of necklaces consisting of $r$ black beads and $k-r$ white beads, where the symmetry group of the necklace
is the cyclic group of order $k$. It is
\[
C(n)=\frac{1}{k}\sum_{d\vert(r,k)}\varphi(d)\binom{k/d}{r/d},
\]
where $(r,k)$ is the greatest common divisor of $r$ and $k$ and $\varphi(d)$ is the Euler $\varphi$-function, i.e., the number of
positive integers $\leq d$ and relatively prime to $d$.
\end{theorem}

\begin{proof}
As in the proof of Theorem \ref{theorem of Bulgarian solitaire}, we shall follow the potential energy $U({\mathcal B}^s(\lambda))$
of the partitions ${\mathcal B}^s(\lambda)$, $s=0,1,2\ldots$.
The minimum of the potential energy will be reached when the first $k-1$ levels in the cradle interpretation of
the diagram $[{\mathcal B}^s(\lambda)]$ are filled in with boxes, and there are no boxes in the $k+1$-st level. We shall identify
such a partition $[{\mathcal B}^s(\lambda)]$ of minimal potential energy with the necklace with $k$ beads, where the $i$-th bead
is black if there is a box in the $i$-th place of the $k$-th level of the diagram, and is white if the $i$-th place is empty,
see Fig. 10.
\begin{figure}[htb]
\begin{center}
\[
\begin{turn}{45}
\epsfxsize=2cm
\epsfbox{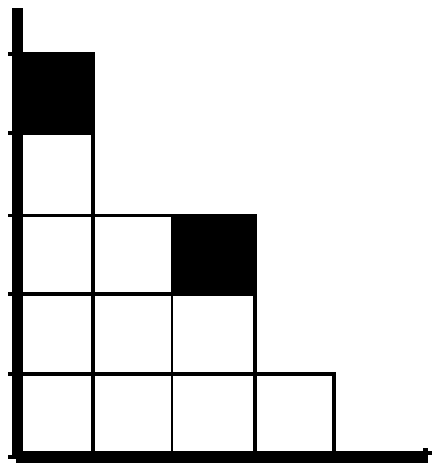}
\end{turn}
\stackrel{\Rightarrow}{\begin{matrix}\phantom{\Downarrow}\\
\phantom{\Downarrow}\\
\phantom{\Downarrow}\\
\end{matrix}}
\begin{turn}{45}
\epsfxsize=2cm
\epsfbox{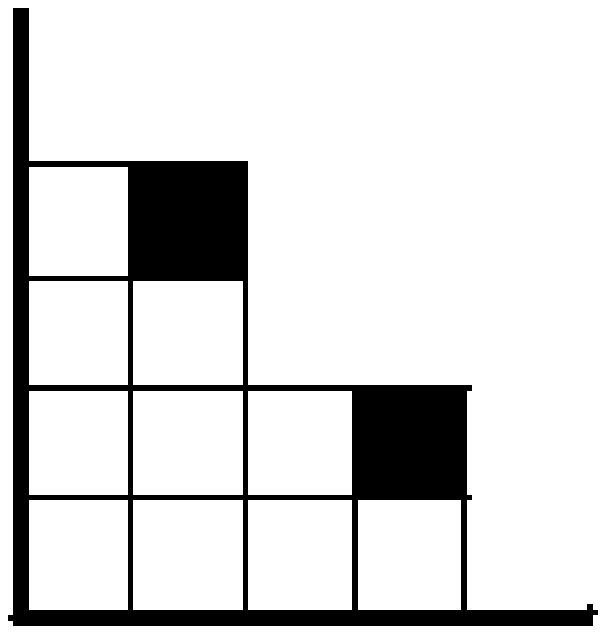}
\end{turn}
\stackrel{\Rightarrow}{\begin{matrix}\phantom{\Downarrow}\\
\phantom{\Downarrow}\\
\phantom{\Downarrow}\\
\end{matrix}}
\begin{turn}{45}
\epsfxsize=2cm
\epsfbox{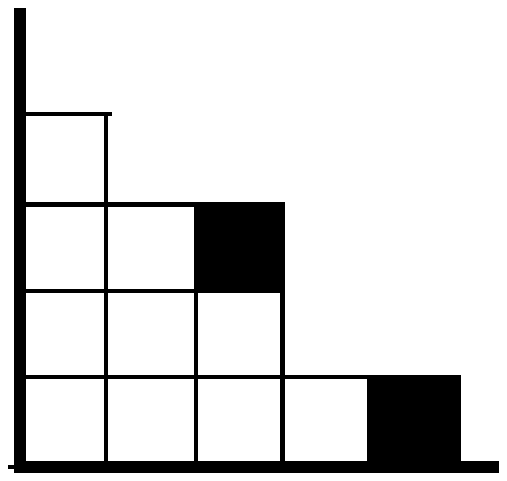}
\end{turn}
\stackrel{\Rightarrow}{\begin{matrix}\phantom{\Downarrow}\\
\phantom{\Downarrow}\\
\phantom{\Downarrow}\\
\end{matrix}}
\begin{turn}{45}
\epsfxsize=2cm
\epsfbox{12_to_rotate_1.eps}
\end{turn}
\]
\[
\Updownarrow
\]
\epsfxsize=5cm
\epsfbox{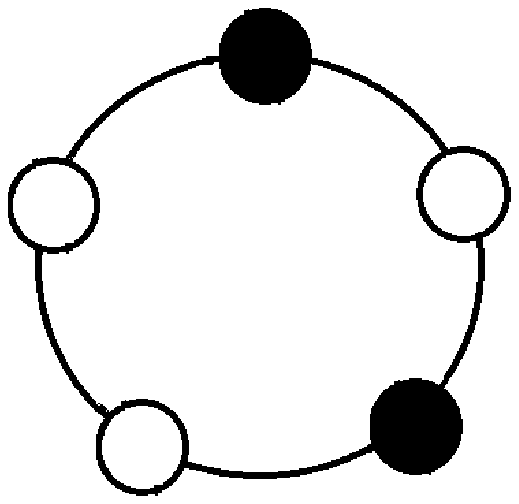}\\
Fig. 10. The cycle generated by (5,3,3,1) and the corresponding necklace.
\end{center}
\end{figure}
\noindent Since the application of the operator $\mathcal B$ moves to the right with period $k$ the boxes of the $k$-th level, we obtain that this
corresponds to the clockwise rotation of the necklace by $360^{\circ}/k$.
The number of necklaces with $r$ black and $k-r$ white beads can be obtained as in \cite{Br} and \cite{AD} as an easy application of the
P\'olya enumeration theorem, see \cite{Bi, DH}.
\end{proof}

In the case of triangular $n=k(k+1)/2$, already Toom \cite{T} raised the problem {\it to determine the longest path in the graph of ${\mathcal P}(n)$
to reach the stable partition $\sigma=(k,k-1,\ldots,2,1)$}.
He showed that, {\it starting from the partition $\tau=(k-1,k-1,k-2,k-3,k-4,\ldots,3,2,1,1)$, the minimal $s$
with the property $\sigma={\mathcal B}^s(\tau)$ is $s=k(k-1)$}. Knuth \cite{HKn} checked the equality
\[
\sigma=(k,k-1,\ldots,2,1)={\mathcal B}^{k(k-1)}(\lambda),\quad \lambda\in {\mathcal P}(k(k+1)/2),
\]
for $k\leq 5$. He asked his students to write a computer program to check it for $k\leq 10$ {\it and conjectured that this holds for any $k$}.
The conjecture of Knuth was proved by Igusa \cite{I} and Bentz \cite{Be}. Bentz also established the following interesting property of the partition
$\tau$ considered by Toom:
{\it The partitions ${\mathcal B}^i(\tau)$ and ${\mathcal B}^{k(k-1)-i-1}(\tau)$ are conjugate for $i=0,1,2,\ldots,k(k-1)-1$.}
This means that the related Young diagrams are obtained by reflection
with respect to the bisectrix from the origin of the first quadrant of the coordinate plane,
i.e., the lengths of the columns of one diagram are equal to the lengths of the rows of the other.
The general case of an arbitrary $n$ was studied by Etienne \cite{Et}.

In the theory of cellular automata, a {\it Garden of Eden configuration} is a configuration that cannot appear on the lattice after one time step,
no matter what the initial configuration. In other words, these are the configurations with no predecessors.
The terminology comes from the foundational paper \cite{M} by analogy with
the concept of the Garden of Eden which, following Semitic religions, was created out of nowhere.
Hopkins and Jones \cite{HJ} studied the {\it Garden of Eden partitions} ({\it GE-partitions}) defined by the property that they do not belong to the image
${\mathcal B}({\mathcal P}(n))$. It has turned out that {\it each cycle in the oriented graph of ${\mathcal P}(n)$ can be reached from a
GE-partition}. We shall mention only the following easy property and refer to \cite{HJ, HK, H2, HS} for more details and further developments.

\begin{proposition}\label{GE-partitions}
A partition $\lambda=(\lambda_1,\ldots,\lambda_s)\vdash n$, $\lambda_s>0$, is a GE-partition if and only if $\lambda_1<s-1$.
\end{proposition}

\section{Generalizations}

Before the paper by Hopkins \cite{H1}, only pieces of the history of the Bulgarian solitaire were known by the large mathematical community.
A couple of times the solitaire was rediscovered or called with other names. The case for triangular $n$ is known also as the {\it Karatsuba solitaire}.
(Since Karatsuba brought it to Bulgaria some people claimed that he invented the puzzle.) We shall discuss several generalizations
of the Bulgarian solitaire which have been studied in the literature.

\subsection{Real life interpretation of the Bulgarian solitaire.}
Discrete dynamical systems often have economic or biological interpretations.
The Bulgarian solitaire reflects the following situation from the real life.
{\it Consider a company consisting
of a number of departments. The Board of Directors decides
to create a new department, but does not want to increase the total number of employees.
So, the Board takes a member from the existing departments and move the person to the new department.}
If we assume that the number of cards in the piles
is equal to the number of persons in the departments,
the Bulgarian solitaire corresponds to the ``greediest'' case, when
the new department is formed by taking a
person from each department of the company.

\subsection{Austrian solitaire.}
Inspired by a discussion on the so-called Austrian school of capital theory, Akin and Davis \cite{AD} introduced the {\it Austrian solitaire}
which has the following economic interpretation. {\it A company has several machines. Each machine has, when new, a life of exactly $L$ years.
Each year for each machine on line the company deposits $1/L$ of its cost into the bank as a sinking fund.
Then it buys as many new machines as it can afford, and the remaining funds are left in the bank until next year.}
Now, {\it take a pack of cards and divide it in piles in such a way that each pack contains not more than $L$ cards.}
Think of the piles as machines. The number of the cards is equal to the number of productive years left for a particular machine.
{\it One of the piles is specific.} It is the bank and does not correspond to a machine. {\it Each move of the solitaire consists of two steps.
In the first step we remove one card from each of the ordinary piles} (the machines have one year less to live)
{\it and add the cards to the pile of the bank. In the second step we take $L$ cards from the bank and form a new ordinary pile of size $L$}
(we buy a new machine) {\it and continue this process until the bank contains $<L$ cards. The problem is to describe the cycles of the corresponding
dynamical system.} For enumeration problems related with the Austrian solitaire see the Baccalaureate Degree Thesis of Bastola \cite{Ba}.

\subsection{Carolina solitaire.}
When visiting the University of South Carolina, Columbia, Andrey Andreev from the Institute of Mathematics and Informatics
at the Bulgarian Academy of Sciences introduced a new ordered variation of the Bulgarian solitaire called the {\it Carolina solitaire.
The game begins with $n$ cards divided into a row of piles of sizes $\alpha_1,\ldots,\alpha_c$, $\alpha_1+\cdots+\alpha_c=n$, $\alpha_i>0$.}
Hence we work with {\it compositions}, i.e., ordered systems of positive integers,
$(\alpha_1,\ldots,\alpha_c)$ rather than with partitions $\lambda=(\lambda_1,\ldots,\lambda_c)\vdash n$.
{\it The move consists of removing one card from each pile, and then placing these $c$ cards
in a pile ahead of the others. Any exhausted pile} ({\it of size} 0) {\it is ignored and
only nonempty piles are considered.} In other words, {\it if $\mathcal C$ is the operator of the Carolina solitaire, then, up to the ignored empty
piles} (obtained for $\alpha_i=1$), {\it the action is defined by}
\[
{\mathcal C}(\alpha_1,\ldots,\alpha_c)=(c,\alpha_1-1,\ldots,\alpha_c-1), \quad \alpha_i>0.
\]
{\it For a triangular number $n=k(k+1)/2$
this new game also appears to arrive at a stable division, with piles of sizes $k, k-1,\ldots,2,1$.}
Griggs and Ho \cite{GH} {\it derived upper and lower bounds for the maximum number of moves required to reach a cycle
of the graph of the compositions of $n$.}
See also the Master Thesis of Tambellini \cite{Ta} for other properties of the Carolina solitaire.

\subsection{Montreal solitaire.}
It is suggested by Cannings and Haigh \cite{CH}. The positions are compositions $\alpha=(\alpha_1,\ldots,\alpha_c)$ of nonnegative integers.
Identifying the compositions $(\alpha_1,\ldots,\alpha_c)$, $(0,\alpha_1,\ldots,\alpha_c)$, and $(\alpha_1,\ldots,\alpha_c,0)$, we may consider
only the case when $\alpha_1$ and $\alpha_c$ are positive. The successor rule $\mathcal M$ of the {\it Montreal solitaire} is defined in the following
way. {If all $\alpha_i$ in $\alpha=(\alpha_1,\ldots,\alpha_c)$ are positive, then
\[
{\mathcal M}(\alpha)=(\alpha_1-1,\ldots,\alpha_c-1,c).
\]
Then we extend the action of $\mathcal M$ inductively. If
\[
\alpha=(\beta,\underbrace{0,\ldots,0}_{r\text{ times}},\gamma)=(\beta,0^r,\gamma),
\quad \beta=(\beta_1,\ldots,\beta_c),\beta_i>0,\gamma=(\gamma_1,\ldots,\gamma_d),
\]
then
\[
{\mathcal M}(\alpha)=({\mathcal M}(\beta,0^{r-1},{\mathcal M}(\gamma)),
\]
keeping the $0$'s in the beginning of ${\mathcal M}(\gamma)$.}
For example,
\[
{\mathcal M}(\text{\bf 1},0,\text{\it 2})=(\text{\bf 0,1},\text{\it 1,1})=(0,1,1,1)(=(1,1,1))
\]
and
\[
{\mathcal M}(\text{\bf 1,2},0,{\it 1,0,2})=(\text{\bf 1,2},\text{\it 0,1,1,1}).
\]
Another example is
\[
(3,2,2)\stackrel{\mathcal M}{\to}(2,1,1,3)\stackrel{\mathcal M}{\to}(1,0,0,2,4)\stackrel{\mathcal M}{\to}(1,0,1,3,2)\stackrel{\mathcal M}{\to}(1,0,2,1,3)
\stackrel{\mathcal M}{\to}(1,1,0,2,3)
\]
and one can check that ${\mathcal M}^{18}(3,2,2)=(3,2,2)$.
In contrast to the Bulgarian solitaire, {\it in the Montreal solitaire each position $\alpha$ has a unique predecessor ${\mathcal M}^{-1}(\alpha)$
and there exists a positive integer $p$ such that ${\mathcal M}^p(\alpha)=\alpha$. Hence the set of compositions $(\alpha_1,\ldots,\alpha_c)$
of $n$ with positive $\alpha_1$ and $\alpha_c$ is a union of disjoint cycles.} We refer to \cite{CH} for more properties of the game.

\subsection{Other discrete generalizations.}
There are also several other games motivated by the Bulgarian solitaire.
As in the case of the regular Bulgarian solitaire, the problems studied concern the type and the number of cycles,
the Garden of Eden positions, etc. We shall list a couple of generalizations.

Locke \cite{L} invented the {\it Red-green Bulgarian solitaire} where the cards are colored in two colors, red and green, and the moves depend
on the existence of green cards in each pile.

Grensj\"o \cite{Gr} studied the {\it Three-dimensional Bulgarian solitaire}.
The idea is {\it to define the game on plane partitions,
which can be visualized using three-dimensional Young diagrams.}

\"Ohman \cite{O} considered two generalizations: the {\it Dual Bulgarian solitaire} and the {\it Multiplayer Bulgarian solitaire}.
In the dual game {\it the piles are ordered in nonincreasing order. In each move the largest pile is removed and its cards are distributed
to the remaining piles one by one, from larger to smaller, with any excessive blocks forming piles of size} 1.
For example, the partition $({\bf 4},4,3,2,2,1,1)$ goes to $(5,4,3,3,1,1)=(4+{\bf 1},3+{\bf 1},2+{\bf 1},2+{\bf 1},1+{\bf 0},1+{\bf 0})$, and
$({\bf 6},6,3,2,1)$ goes to $(7,4,3,2,1,1)=(6+{\bf 1},3+{\bf 1},2+{\bf 1},1+{\bf 1},0+{\bf 1},0+{\bf 1})$.
{\it In the regular Bulgarian solitaire we may assume that each part of a partition corresponds to the income of a citizen or a company.
Then the Government collects the same taxes from each person and each company and uses the collected money to make a new company.}
In the dual game, see the comments in \cite{Gr}, {\it one applies the principles of Robin Hood: taking from the rich and giving to the
poor.} It can be shown that {\it if Robin Hood continues to take from the rich and give
to the poor, then the distribution of fortune in his community will become close to triangular.} The explanation is simple. {\it If
one translates the game in the cradle model, it is easy to see that the dual Bulgarian solitaire is equivalent to the regular one.}
Bouchet \cite{Bo1, Bo2}, see also Bruhn \cite{Bru}, established that {\t the dual Bulgarian solitaire corresponds to the
old African game {\it Owari} which consists of cyclically ordered pits that are filled with pebbles.
In a sowing move all the pebbles are taken out of one pit and distributed one by one in subsequent pits.
Repeated sowing will give rise to recurrent states of the owari.}

One can interpret the multiplayer game in the following way.
{\it Several players sitting around a circular table play the Bulgarian
solitaire. All players remove one card from each of their piles at the same time and then pass
this new pile to the player on their right.} In other words, if we have a collection of partitions
\[
\lambda^{(1)}=(\lambda_1^{(1)},\ldots,\lambda_{c_1}^{(1)}),\ldots,\lambda^{(s)}=(\lambda_1^{(s)},\ldots,\lambda_{c_s}^{(s)}),
\quad \lambda_{c_i}^{(i)}>0,
\]
the move sends $\lambda^{(i)}$ to $(c_{i-1},\lambda_1^{(i)}-1,\ldots,\lambda_{c_i}^{(i)}-1)$, where by convention $c_0=c_n$
and the parts of the image of $\lambda^{(i)}$ are rearranged in nonincreasing order if necessary.

Servedio and Yeh \cite{SY} suggested a game which can be interpreted in the following way.
{\it There are $c$ players sitting around a circular table. The $i$-th player has $\alpha_i$ cards.} (We consider circular compositions
on $n$, identifying $\alpha=(\alpha_1,\ldots,\alpha_c)$ and $(\alpha_c,\alpha_1,\ldots,\alpha_{c-1})$.)
{\it The move consists of the following simultaneous actions of the players. The $i$-th one takes one's cards and distributes them clockwise,
to oneself and to the following $\alpha_i-1$ players.}

Janetzko in his Ph. D. Thesis \cite{J}
considered a similar game {\it with $c$ players around a circular table and with total number of $n$ cards.
A pointer points one of the persons} (e.g., the $i$-th one) {\it who takes all cards from his or her pile and distributes
them to all players on the right, giving one card to the $(i+1)$-th player, one card to the $(i+2)$-th player, etc.} (the addition is modulo $c$).
{\it At the end the pointer points at the player that receives the last card.
Repeating this procedure gives a periodic sequence of pointer positions.}
The thesis studies the question which periodic sequences can be realized as such pointer sequences.
It is interesting to mention that the problem is reduced to the investigation of an inhomogeneous linear system of equations.
Then the author applies the Perron-Frobenius theorem, \cite{P} and \cite{F},
{\it which asserts that a real square matrix with positive entries has a unique
largest real eigenvalue and that the corresponding eigenvector has strictly positive components,
and also asserts a similar statement for certain classes of nonnegative matrices.}

\subsection{Stochastic Bulgarian solitaires.}
There are many possible ways to formulate stochastic versions of the
Bulgarian solitaire. Popov \cite{Po} introduced his {\it Random Bulgarian solitaire}. As in the regular Bulgarian solitaire,
{a deck of $n$ cards is divided into several piles. Then one fixes a number $p\in (0,1]$ and, for each pile,
one leaves it intact with probability $1-p$ and removes one card from the pile with probability $p$, independently of the other piles.
The cards that are removed are collected to form a new pile.} For $p=1$ this is the regular, or deterministic, Bulgarian solitaire.
{\it The model with parameter $0<p<1$  is a discrete-time irreducible
and aperiodic Markov chain on the space of unordered partitions of $n$.
For the stationary measure of the game Popov proves that most of its mass is concentrated on {\rm (}roughly{\rm )}
triangular configurations of a certain type.}
Eriksson and Sj\"ostrand \cite{ES} showed that {\it the random Bulgarian solitaire
can be interpreted as a birth-and-death process on Young diagrams.}

Recently Eriksson, Jonsson, and Sj\"ostrand \cite{EJS} introduced another {\it Stochastic Bulgarian solitaire}.
They assume that the selection acts on the cards rather than on the piles: {\it When forming a new pile by picking cards
from the old piles, every card is picked with a fixed probability $0 < p < 1$,
independently of all other cards.} They establish a surprising fact. {\it The solitaire
is not drawn to triangular configurations but to an exponential shape.}


\begin{thebibliography}{99}

\bibitem{AD}
E. Akin, M. Davis,
Bulgarian solitaire,
Amer. Math. Monthly  92  (1985), No. 4, 237--250.

\bibitem{Ba}
K. Bastola,
Enumeration of Austrian Solitaire,
Baccalaureate Degree Thesis,
Saint Peter's University, Jersey City, NJ, USA, 2012.\\
http://librarydb.saintpeters.edu:8080/xmlui/bitstream/handle/123456789/30/Kapil.pdf?\\
sequence=1

\bibitem{Be}
H.-J. Bentz,
Proof of the Bulgarian Solitaire conjectures,
Ars Combin. 23 (1987), 151-170.

\bibitem{Bi}
N. L. Biggs,
Discrete Mathematics, 2nd ed., Oxford Univ. Press, Oxford, 2002.

\bibitem{B1}
B. Bojanov,
Problem Solution 4 (Bulgarian), In: Obuchenieto po matematica 24 (1981), No. 5, 59-60.

\bibitem{B2}
B. Bojanov,
Email to H. Eriksson, forwarded to V. Drensky, Nov. 22, 2005.

\bibitem{Bo1}
A. Bouchet,
Owari I. Marching groups and periodical queues,\\
http://warimath.free.fr/Documents/OwariI.pdf

\bibitem{Bo2}
A. Bouchet,
Owari II. Marching groups and Bulgarian solitaire,\\
http://warimath.free.fr/Documents/OwariII.pdf

\bibitem{Br}
J. Brandt,
Cycles of partitions,
Proc. Amer. Math. Soc. 85 (1982), No. 3, 483-486.

\bibitem{Bru}
H. Bruhn,
Periodical states and marching groups in a closed Owari,
Discrete Math. 308 (2008), No. 16, 3694-3698.

\bibitem{CH}
C. Cannings, J. Haigh,
Montreal solitaire,
J. Combin. Theory Ser. A 60 (1992), No. 1, 50-66.

\bibitem{DH}
L. L. Dornhoff, F. E. Hohn,
Applied Modern Algebra,
Macmillan Publishing Co., Inc., New York, Collier Macmillan Publishers, London, 1978.

\bibitem{E}
H. Eriksson,
Bulgarisk patiens (Swedish),
Elementa 64 (1981), No. 4, 186-188.

\bibitem{EJS}
K. Eriksson, M. Jonsson, J. Sj\"ostrand,
The limit shape of a stochastic Bulgarian solitaire,
arXiv: 1309.2846v1 [math.PR].

\bibitem{ES}
K. Eriksson, J. Sj\"ostrand,
Limiting shapes of birth-and-death processes on Young diagrams,
Adv. Appl. Math. 48 (2012), No. 4, 575-602.

\bibitem{Et}
G. Etienne,
Tableux de Young et Solitaire Bulgare,
J. Comb. Theory, Ser. A 58 (1991), No. 2, 181-197.

\bibitem{F}
G. Frobenius,
\"Uber Matrizen aus nicht negativen Elementen,
Sitzungsber. K\"onigl. Preuss. Akad. Wiss. (1912), 456–477.

\bibitem{G}
M. Gardner,
Bulgarian solitaire and other seemingly endless tasks,
Sci. Amer. 249 (1983), 12-21.

\bibitem{Gr}
A. Grensj\"o,
Bulgarian Solitaire in Three Dimensions,
Research Academy for Young Scientists, March 24, 2013.
Royal Institute of Technology, Stockholm.\\
http://www.raysforexcellence.se/wp-content/uploads/2013/06/Bulgarian-Solitaire-in-Three-Dimensions-Anton-Grensjo.pdf

\bibitem{GH}
J. R. Griggs, C.-C. Ho,
The cycling of partitions and composition under repeated shifts,
Adv. Appl. Math. 21 (1998), No. 2, 205-227.

\bibitem{HKK}
T. A. Hart, G. J. H. Khan, M. R. Khan,
Revisiting Toom's proof of Bulgarian solitaire,
Ann. Sci. Math. Qu\'e. 36 (2012), No. 2, 477-486.

\bibitem{HKn}
J. D. Hobby, D. Knuth,
Problem 1: Bulgarian Solitaire,
In: A Programming and Problem-Solving Seminar.
Department of Computer Science, Stanford University, 1983 (December), 6-13.\\
http://i.stanford.edu/pub/cstr/reports/cs/tr/83/990/CS-TR-83-990.pdf

\bibitem{H1}
B. Hopkins,
30 years of Bulgarian solitaire,
College Math. J. 43 (2012), No. 2, 135-140.

\bibitem{H2}
B. Hopkins,
Column-to-row operations on partitions: the envelopes,
Combinatorial Number Theory, 65-76, Walter de Gruyter, Berlin, 2009.

\bibitem{HJ}
B. Hopkins, M. A. Jones,
Shift-induced dynamical systems on partitions and compositions,
Electron. J. Comb. 13 (2006), No. 1, Research paper R80, 19 p.

\bibitem{HK}
B. Hopkins, L. Kolitsch,
Column-to-row operations on partitions: Garden of Eden partitions,
Ramanujan J. 23 (2010), No. 1-3, 335-339.

\bibitem{HS}
B. Hopkins, A. Sellers,
Exact enumeration of Garden of Eden partitions,
Combinatorial Number Theory, 299-303, de Gruyter, Berlin, 2007.

\bibitem{I}
K. Igusa,
Solution of the Bulgarian solitaire conjecture,
Math. Mag. 58 (1985), 259-271.

\bibitem{J}
H.-D. Janetzko,
Realisierung von Zeigerperioden f\"ur einen Verteilungsalgorithmus nach Art des Bulgarischen Solit\"ars,
TH Aachen, Math.-Naturwiss. Fak., Aachen, 1996.

\bibitem{L}
S. C. Locke,
Red-green Bulgarian solitaire,\\
http://math.fau.edu/locke/courses/Rec-Math/RedGreenBulgarianSolitaire.htm

\bibitem{Me}
R. Me\v{s}trovi\'c,
An inductive proof of a result about Bulgarian solitaire,
Ars Comb. 95 (2010), 65-70.

\bibitem{M}
E. F. Moore. Machine models of self-reproduction, 
Proc. Sympos. Appl. Math. 14 (1962), 17-33.

\bibitem{O}
E. \"Ohman,
Multiplayer Bulgarian Solitaire,
Research Academy for Young Scientists, July 11, 2012.
Royal Institute of Technology, Stockholm.\\
http://www.raysforexcellence.se/wp-content/uploads/2013/02/Multiplayer-Bulgarian-Solitaire.pdf

\bibitem{P}
O. Perron,
Zur Theorie der Matrices,
Math. Ann. 64 (1907), 248-263.

\bibitem{Po}
S. Popov,
Random Bulgarian solitaire,
Random Structures Algorithms 27 (2005), No. 3, 310-330.

\bibitem{SY}
R. Servedio, Y.-N. Yeh,
A bijective proof on circular compositions,
Bull. Inst. Math., Acad. Sin. 23 (1995), No. 4, 283-293.

\bibitem{Ta}
L. Tambellini,
Sistemas Din\^amicos Finitos: Paci\^encia B\'ulgara
(Shift em Parti\c{c}\~oes e Composi\c{c}\~oes C\'{\i}clicas),
M. Sci. Thesis, Universidade Estadual Paulista
``J\'ulio De Mesquita Filho'', S\~ao Jos\'e do Rio, SP, Brazil, 2013.\\
http://base.repositorio.unesp.br/bitstream/handle/11449/94253/tambellini{\_}l{\_}me{\_}sjrp.pdf?\\
sequence=1{\&}isAllowed=y

\bibitem{T}
A. L. Toom,
Problem Solution M655 (Russian), In: Kvant 12 (1981), No. 7, 28-30.

\end{thebibliography}
\end{document}